\documentclass[12pt]{article}

\usepackage{pbmath}
\usepackage{epsf}
\usepackage{psfrag}

\title{Detecting automorphic orbits in free groups}
\date{\today}
\author{Peter Brinkmann}

\def\vv#1{{\cal V}(#1)}
\def\ee#1{{\cal E}(#1)}

\begin{document}
\maketitle

\begin{abstract}
We present an effective algorithm for detecting automorphic orbits in
free groups, as well as a number of algorithmic improvements of train
tracks for free group automorphisms.
\end{abstract}

\section*{Introduction}

The following theorem is the main result of this paper.

\begin{thm}\label{mainthm}
Let $\phi$ be an automorphism of a finitely generated free group $F_n$.
\begin{itemize}
\item There exists an explicit algorithm that, given two elements $u,
v\in F_n$, decides whether there exists some exponent $N$ such that
$u\phi^N=v$.
\item There exists an explicit algorithm that, given two elements $u,
v\in F_n$, decides whether there exists some exponent $N$ such that
$u\phi^N$ is conjugate to $v$.
\end{itemize}
If such an exponent $N$ exists, then the algorithms will compute $N$ as
well.  The words $u, v$ are specified as words in the generators of
$F_n$, and $\phi$ is specified in terms of the images of generators.
\end{thm}

The results in this paper was motivated by work that first appeared in
\cite{rect}.  \thmref{mainthm} plays a role in the computation of fixed
subgroups of free group automorphisms \cite{fixedsubgroup}, and it
constitutes one part of the recent solution of the conjugacy problem in
free-by-cyclic groups due to Bogopolski, Maslakova, Martino, and Ventura
\cite{conjprob}.

Our main technical tool is an algorithmic extension of the theory of
relative train track maps \cite{hb1, hb2}.  Specifically, we present
algorithmic (and possibly even practical) ways of finding {\em
efficient} relative train track maps that share many the properties of
{\em improved} relative train track maps as introduced (in a
nonconstructive fashion) in \cite{tits1}.  

One intriguing aspect of our argument is that it suggests that the
detection of orbits in free groups and the computation of efficient maps
are closely related problems.  Orbit detection and computation of
efficient maps leapfrog each other, with orbit detection providing a
crucial step in the computation of efficient maps, and efficient maps
enabling the detection of orbits.

In \secref{bccsec}, we review well-known results on homotopy
equivalences of finite graphs, with an emphasis on computational aspects
of the constants involved.  \secref{trainsec} contains a brief review of
the theory of relative train track maps, including first steps towards
improvements.  \secref{improvesec} contains the first part of our
construction of efficient train track maps.  \secref{pathsec} presents
an algorithm that detects orbits of paths, and \secref{fixedsec} builds
upon the results of \secref{improvesec} and \secref{pathsec} to provide
the last, and most difficult, step in our construction of efficient maps,
the detection of fixed points of certain lifts of homotopy equivalences
of finite graphs.  Finally, in \secref{mainsec}, we translate our
results from the realm of homotopy equivalences of graphs to the realm
of automorphisms of free groups.

I would like to express my gratitude to Oleg Bogopolski and Armando
Martino for their hospitality, encouragement, and many helpful
discussions.


\section{Quasi-isometries and bounded cancellation}\label{bccsec}

The results in this section are well-known.  We list them here, with
detailed proofs, because explicit computations of the constants involved
do not seem to appear in the literature.

Let $f\co G\rightarrow H$ be a homotopy equivalence of finite connected
graphs, which we equip with the usual path metric (denoted by $|.|$),
and let $g: H\rightarrow G$ be a homotopy inverse of $f$.\footnote{Given
$f$, we can easily compute $g$ (see, for instance, \cite{ls}).}  We denote the
set of vertices of $G$ by $\vv G$ and the set of edges by $\ee G$.
Throughout this paper, we only consider homotopy equivalences that map
vertices to vertices and edges to edge paths of constant (but not
necessarily identical) speed.  We may assume that there exists some
vertex $\bar v_0$ such that $\bar v_0fg=\bar v_0$.

Let $\tilde f\co \tilde G\rightarrow \tilde H$ be a lift of $f$ to the
universal covers, with a lift $v_0$ of $\bar v_0$.  Given $x, y\in
\tilde G$, we denote the unique geodesic path connecting $x$ and $y$ by
$[x, y]$.  For brevity, we write $|x, y|$ for $|[x, y]|$.  We define
$[x, y]\tilde f=[x\tilde f, y\tilde f]$.\footnote{Note that the
composition of the path $[x, y]$ and $\tilde f$ is not, in general, an
immersion.  The path $[x\tilde f, y\tilde f]$ is the unique immersed
path that is homotopic relative endpoints to this composition.}

The lift $\tilde f$ extends to a homeomorphism of the boundaries
$\partial\tilde G, \partial\tilde H$.  Let $\tilde g\co \tilde H
\rightarrow \tilde G$ be a lift of $g$ such that satisfies $v_0\tilde
f\tilde g=v_0$, and note that $\tilde f\tilde g$ induces the
identity on $\partial\tilde G$.

Arguments involving universal covers are generally nonconstructive.  The
universal cover of a finite connected graph, however, is a tree, and we
can construct arbitrarily large subtrees as well as partial lifts of
maps to these subtrees, which is enough for the computations we will
encounter.  We describe this construction here, with the tacit
understanding that all computations in universal covers will require it
as a preliminary step.

\begin{cons}\label{covercons}
Fix some vertex $\bar v_0\in G$.  Let $v_0\in\tilde G$ be a lift of
$\bar v_0$ and $w_0\in\tilde H$ a lift of $\bar w_0=\bar v_0f$.  We let
$T_0=\{v_0\}$ and $U_0=\{w_0\}$ and define $\tilde f_0\co
T_0\rightarrow U_0$ in the only possible way.

Now, suppose we have subtrees $T_0\subseteq T_1\subset\tilde G$ and
$U_0\subseteq U_1\subset\tilde H$ as well as a partial lift $\tilde
f_1\co T_1\rightarrow U_1$, i.e., $\tilde f|_{T_1}=\tilde f_1$.  Our
goal is to enlarge $T_1$ and $U_1$ and extend $\tilde f_1$ accordingly.

There is a bijective relationship between vertices of $\tilde G$ and
edge paths in $G$ originating at $\bar v_0$.\footnote{In our
computations, we will always be given such paths for those vertices of
$\tilde G$ that we are interested in.} Let $\rho$ be an edge path in $G$
originating at $\bar v_0$.  We want to construct $T_2$ so that it
contains a lift of $\rho$.  To this end, starting with $v_0$ and the first
edge of $\rho$, we keep track of a current vertex $v$ and a current edge
$E$.  If $T_1$ already contains an edge $E'$ originating at $v$ that
projects to $E$, we make the other endpoint of $E'$ our current vertex
and move on to the next edge of $\rho$.  If no such edge exists, we
attach a new edge at $v$ and map it to $E$.  Then we move on to the
terminal endpoint of the new edge and the next edge in
$\rho$.\footnote{An alternative approach is to attach an entire lift of
$\rho$ at $v_0$ and then fold as necessary \cite{stallings}.}

Now, for each vertex $v$ of $T_2\setminus T_1$, we compute the image
$\rho_v$ of the path $[v_0, v]$ in $G$, and we construct a lift of
$\rho_v f$ to the universal cover.  Like before, we construct $U_2$ by
extending $U_1$ such that it includes these lifts, obtaining a larger
subtree of $\tilde H$ as well as a partial lift $\tilde f_2\co
T_2\rightarrow U_2$.

Proceeding in this fashion, we can build arbitrarily large subtrees of
$\tilde G$ and $\tilde H$ along with partial lifts of $f$.  If $G=H$, we
can and will arrange that $T_2\subseteq U_2$.

\end{cons}

The lift $\tilde f$ is a {\em quasi-isometry}, i.e., there exist constants
$K_f, D_f$ such that for all $x, y\in \tilde G$, we have
\begin{equation}\label{qi}
\frac{|x,y|}{K_f}-D_f\leq |x\tilde f, y\tilde f| \leq K_f|x, y|+D_f.
\end{equation}
We need to compute suitable constants $K_f, D_f$.  To this end, define
the {\em size} of $f$ to be $S_f=\max_{E\in\ee G}\{|Ef|\}$.

\begin{lem}\label{bfglem}
We can compute a number $B_{fg}$ satisfying
\[
  B_{fg}\geq \max_{x\in \tilde G} \{|x, x\tilde f\tilde g|\}.
\]
\end{lem}

\begin{proof}
We first compute $B=\max_{v\in\vv{\tilde G}}\{|v, v\tilde f\tilde g|\}$.
Let $\gamma$ be a deck transformation of $\tilde G$.  Since $\tilde
f\tilde g$ extends to the identity on $\partial\tilde G$, we have
$\gamma\tilde f\tilde g=\tilde f\tilde g\gamma$.

For $v\in\vv{\tilde G}$, we have $|v\gamma, v\gamma\tilde f\tilde
g|=|v\gamma, v\tilde f\tilde g\gamma|=|v, v\tilde f\tilde g|$, so that
we only need to check one representative of each orbit of vertices.
The distance $|v, v\tilde f\tilde g|$ is the length of the path obtained
by concatenating $[v, v_0]$ and $[v_0, v\tilde f\tilde g]$ and
tightening.  Hence, we can compute $B$.

Now consider some point $x\in\tilde G$.  Then there exists some vertex
$v\in\vv{\tilde G}$ such that $|x, v|<1$, so that
$|x, x\tilde f\tilde g|\leq 1+|v, v\tilde f\tilde g|+S_{fg}\leq
1+B+S_{fg}$.
\end{proof}

\begin{lem}\label{qilem}
\ineqref{qi} holds with $K_f=\max\{S_f, S_g\}$ and
$D_f=\frac{2B_{fg}}{K_f}$.
\end{lem}

\begin{proof}
Let $x, y\in\tilde G$.  By definition of $K_f$, we have
$|x\tilde f, y\tilde f|\leq K_f|x, y|$, so that the upper bound in
\ineqref{qi} holds.

Similarly, we have $|x\tilde f\tilde g, y\tilde f\tilde g|\leq
K_f|x\tilde f, y\tilde f|$.  The triangle inequality implies that
$|x, y|\leq |x, x\tilde f\tilde g|+|x\tilde f\tilde g, y\tilde f\tilde
g|+|y\tilde f\tilde g, y|\leq |x\tilde f\tilde g, y\tilde f\tilde
g|+2B_{fg}\leq K_f|x\tilde f, y\tilde f| + 2B_{fg}$.  We conclude that $|x, y|-2B_{fg}\leq K_f|x\tilde f,
y\tilde f|$, and the claim follows.
\end{proof}

Thurston's {\em Bounded Cancellation Lemma} \cite{cooper} is a
fundamental tool in the theory of free group automorphisms.  We present
a proof here because we require an explicit bound on the constant
involved.

Let $p, x, y$ be points in $\tilde G$ and let $\alpha=[p, x]$ and
$\beta=[p, y]$.  We denote the common (possibly trivial) initial segment
of $\alpha$ and $\beta$ by $\alpha\wedge\beta$.  If $\alpha$ is a prefix
of $\beta$, we write $\alpha\leq\beta$.

\begin{lem}[Bounded Cancellation Lemma]\label{bcc}
Let $C_f=(B_{fg}+D_g+S_g)K_g$.  If $|\alpha\wedge\beta|=0$, then
\[
|\alpha\tilde f\wedge\beta\tilde f|\leq C_f.
\]
\end{lem}

\begin{proof}
Let $L=|\alpha\tilde f\wedge\beta\tilde f|$.  \ineqref{qi} implies that
$|(\alpha\tilde f\wedge\beta\tilde f)\tilde g|\geq \frac{L}{K_g}-D_g$,
so that $|\alpha\tilde f\tilde g\wedge\beta\tilde f\tilde g|\geq
\frac{L}{K_g}-D_g-S_g$.  Now \lemref{bfglem} implies that
\[
|\alpha\wedge\beta|\geq \frac{L}{K_g}-D_g-S_g-B_{fg}.
\]
Hence, if $L>C_f$, then $|\alpha\wedge\beta|>0$.
\end{proof}

Finally, we record a basic property of homotopy equivalences of graphs.

\begin{lem}\label{fppreim}
Let $f\co G\rightarrow G$ be a homotopy equivalence of a finite graph.
If $\alpha$ is a path in $G$ whose endpoints are fixed by $f$,
then there exists some path $\beta$ with the same endpoints satisfying
$\beta f=\alpha$.
\end{lem}

\begin{proof}
Let $v$ be the initial endpoint of $\alpha$.  Then there exists some
loop $\sigma$ based at $v$ so that $\alpha f$ is homotopic (relative
endpoints) to the concatenation $\sigma\alpha$.  Since $f$ is a homotopy
equivalence, there exists a loop $\sigma'$ satisfying $\sigma'f=\sigma$,
and we conclude that $(\bar{\sigma}'\alpha)f=\alpha$.
\end{proof}


\section{Relative train track maps}\label{trainsec}

In this section, we review the theory of relative train tracks maps
\cite{hb1, dv} as well as first steps towards our take on
improvements of relative train track maps.

Given an automorphism $\phi\in Aut(F)$, we can find a based homotopy
equivalence $f\co G\rightarrow G$ of a finite connected graph $G$ such
that $\pi_1(G)\cong F$ and $f$ induces $\phi$. This observation allows
us to apply topological techniques to automorphisms of free groups. In
many cases, it is convenient to work with outer automorphisms.
Topologically, this means that we work with homotopy equivalences rather
that based homotopy equivalences.

Oftentimes, a homotopy equivalence $f\co G\rightarrow G$ will respect a
{\em filtration} of $G$, i.\ e., there exist subgraphs
$G_0=\emptyset\subset G_1 \subset \cdots \subset G_k=G$ such that for
each filtration element $G_r$, the restriction of $f$ to $G_r$ is a
homotopy equivalence of $G_r$. The subgraph $H_r=\overline{G_r\setminus
G_{r-1}}$ is called the {\em $r$-th stratum} of the filtration.  We say
that a path $\rho$ has {\em nontrivial intersection} with a stratum
$H_r$ if $\rho$ crosses at least one edge in $H_r$.

If $H_r=\{E_1,\cdots,E_m\}$, then the {\em transition matrix} of $H_r$
is the nonnegative $m\times m$-matrix $M_r$ whose $ij$-th entry is the
number of times the $f$-image of $E_j$ crosses $E_i$, regardless of
orientation. $M_r$ is said to be {\em irreducible} if for every tuple
$1\leq i,j \leq m$, there exists some exponent $n>0$ such that the
$ij$-th entry of $M_r^n$ is nonzero.  If $M_r$ is irreducible, then it
has a maximal real eigenvalue $\lambda_r\geq 1$ \cite{gantmacher}.  We
call $\lambda_r$ the {\em growth rate} of $H_r$.

Given a homotopy equivalence $f\co G\rightarrow G$, we can always find a
filtration of $G$ such that each transition matrix is either a zero
matrix or irreducible. A stratum $H_r$ in such a filtration is called
{\em zero stratum} if $M_r$ is a zero matrix. $H_r$ is called {\em
exponential } if $M_r$ is irreducible with $\lambda_r>1$, and it is
called {\em nonexponential} if $M_r$ is irreducible with $\lambda_r=1$.

An unordered pair of edges in $G$ originating from the same vertex is
called a {\em turn}. A turn is called {\em degenerate} if the two edges
are equal.  We define a map $Df\co \{\text{turns in } G\}\rightarrow
\{\text{turns in } G\}$ by sending each edge in a turn to the first edge
in its image under $f$. A turn is called {\em illegal} if its image
under some iterate of $Df$ is degenerate; otherwise, it is called {\em
legal}.

An edge path $\rho=E_1E_2\cdots E_s$ is said to contain the turns
$(E_i^{-1},E_{i+1})$ for $1\leq i <s$; $\rho$ is legal if all its turns
are legal, and it is $r$-legal if $\rho\subset G_r$ and no illegal turn
in $\rho$ involves an edge in $H_r$.

Let $\rho$ be a path in $G$. In general, the composition
$\rho\circ f^k$ is not an immersion, but there is a unique immersion
that is homotopic to $\rho\circ f^k$ relative endpoints. We denote this
immersion by $\rho f^k$, and we say that we obtain $\rho f^k$ from
$\rho\circ f^k$ by {\em tightening}. If $\sigma$ is a circuit in $G$,
then $\sigma f^k$ is the immersed circuit homotopic to $\sigma\circ f^k$.

\begin{thm}[{\cite[Theorem 5.12]{hb1}}]\label{rtt}
Every outer automorphism of $F$ is represented by a homotopy equivalence
$f\co G\rightarrow G$ such that each exponential stratum $H_r$ has the
following properties:
\begin{enumerate}
\item If $E$ is an edge in $H_r$, then the
first and last edges in $Ef$ are contained in $H_r$.
\item If $\beta$ is a nontrivial path in $G_{r-1}$ with endpoints in
$G_{r-1}\cap H_r$, then $\beta f$ is nontrivial.\label{lowerstrat}
\item If $\rho$ is an $r$-legal path, then $\rho f$ is an $r$-legal path.
\end{enumerate}
\end{thm}
We call $f$ a {\em relative train track map}.  A detailed, explict
algorithm for computing relative train track maps appeared in \cite{dv}.

We conclude this section with the introduction of some terminology that
will be needed later.

A path $\rho$ is a {\em (periodic) Nielsen path} if $\rho f^k=\rho$
for some $k>0$. In this case, the smallest such $k$ is the {\em period}
of $\rho$. A Nielsen path $\rho$ is called {\em indivisible} if it cannot
be expressed as a concatenation of shorter Nielsen paths.

A decomposition of a path $\rho=\rho_1\cdot\rho_2\ldots\cdot\rho_s$ into
subpaths is called a {\em $k$-splitting} if $\rho f^k=\rho_1f^k\cdots
\rho_s f^k$, i.e., there is no cancellation between $\rho_if^k$ and
$\rho_{i+1}f^k$ for $1\leq i<s$.  Such a decomposition is a {\em
splitting} if it is a $k$-splitting for all $k>0$. We will also use the
notion of $k$-splittings of circuits
$\sigma=\rho_1\cdot\rho_2\ldots\cdot\rho_s$, which requires, in
addition, that there be no cancellation between $\rho_s f^k$ and $\rho_1
f^k$.

The {\em $r$-length} of a path $\rho$ in $G$, denoted by $|\rho|_r$, is
the number of edges in $H_r$ that $\rho$ crosses.  A path $\rho$ in $G$
is said to be of {\em height $r$} if $\rho$ is contained in $G_r$ but
not in $G_{r-1}$. If $H_r=\{E_r\}$ is a nonexponential stratum, then
{\em basic paths} of height $r$ are of the form $E_r\gamma$ or
$E_r\gamma E_r^{-1}$, where $\gamma$ is a path in $G_{r-1}$.

\begin{defn}
We say that a relative train track map $f\co G\rightarrow G$ is {\em
normalized} if the following properties hold:
\begin{enumerate}
\item For every vertex $v\in\vv{G}$, $vf$ is a fixed vertex of $f$.
\item Every nonexponential stratum $H_r$ contains only one edge $E_r$
and $E_rf=E_ru_r$ for some path $u_r$ in $G_{r-1}$.
\item If $H_r=\{E_r\}$ is a nonexponential stratum, $u_r$ is of
height $s$, and $s<t<r$, then $H_t$ is nonexponential and $u_t$ is also
of height $s$.
\item If $E$ is an edge in an exponential stratum $H_r$, then
$|Ef|_r\geq 2$.
\item Every isolated fixed point of $f$ is a vertex.
\item If $C$ is a noncontractible component of some filtration element
$G_r$, then $C=Cf$.
\end{enumerate}
\end{defn}

\begin{lem}\label{rttnorm}
Every outer automorphism ${\cal O}$ has a positive power ${\cal O}^k$
that is represented by a normalized relative train track map $f\co
G\rightarrow G$.  Both $k$ and $f$ can be computed.
\end{lem}

\begin{proof}
First, we compute a relative train track map $f'\co G'\rightarrow G'$
representing $\cal O$ \cite{hb1, dv}.  We easily read off an exponent
$k$ such that $f'^k$ satisfies the first, fourth, and sixth properties of
normalized maps, and we have $Ef'^k=vEw$ for every edge $E$ in a
nonexponential stratum $H_r$.

After replacing $f$ by a power $f^k$, we may need to refine the
filtration of $G$ because an irreducible matrix may have reducible
powers.  We may also need to permute some filtration elements in order
to achieve the desired alignment of nonexponential strata.

If $v$ is nontrivial and $w$ is trivial, we reverse the orientation of
$E$.  If both $v$ and $w$ are nontrivial, we split $E$ into two edges
$E', E''$ such that $E=\bar{E'}E''$ and $E'f'^k=E'\bar{v}$ and
$E''f'^k=E''w$.

By refining the filtration of $G'$ so that each nonexponential stratum
contains exactly one edge and subdividing at isolated fixed points if
necessary, we obtain a normalized representative $f\co G\rightarrow G$
of ${\cal O}^k$.
\end{proof}

\begin{lem}\label{fixedvertices}
Let $f\co G\rightarrow G$ be a normalized relative train track map with
an exponential stratum $H_r$.  If $C$ is a noncontractible component of
$G_{r-1}$ and $v$ is a vertex in $H_r\cap C$, then $v=vf$.
\end{lem}

\begin{proof}
This argument is contained in the proof of \cite[Theorem~5.1.5]{tits1}.
We repeat it here because it is short.

Let $v$ be a vertex in $H_r\cap C$.  Since $f$ is normalized, we have
$C=Cf$, so that there exists a path $\alpha$ in $C$ that starts at $v$
and ends at $vf$.  The vertex $vf$ is fixed, and there exists some path
$\beta$ in $C$ that starts and ends at $vf$ such that $\alpha f=\beta
f$.  Then $(\alpha\bar{\beta})f$ is trivial, so that $\alpha\bar{\beta}$
is trivial because of the second property of relative train track maps.
\end{proof}

\begin{lem}\label{pgsplit}
Let $f\co G\rightarrow G$ be a normalized train track map with a
nonexponential stratum $H_r$. If $\rho$ is a path in $G_r$, then it
splits as a concatenation of basic paths of height $r$ and paths in
$G_{r-1}$.
\end{lem}

\begin{proof}
This is essentially \cite[Lemma 4.1.4]{tits1}.  The lemma follows
immediately from the second property of normalized train track maps.
\end{proof}

\begin{lem}\label{critlen}
Let $f\co G\rightarrow G$ be a normalized train track map with an
exponential stratum $H_r$.  If $\rho$ is a circuit or edge path of
height $r$ containing an $r$-legal subpath of $r$-length $L>2C_f$ (where
$C_f$ is the bounded cancellation contant of $f$), then $\rho f$
contains an $r$-legal subpath of $r$-length greater than $L$.
\end{lem}

\begin{proof}
This is an immediate consequence of \lemref{bcc} and the fourth property
of normalized maps, which implies $\lambda_r\geq 2$.
\end{proof}

We will need the following consequence of \cite[Proposition~6.2]{pbhyp}.

\begin{lemnp}\label{tricho}
Let $f\co G\rightarrow G$ be a relative train track map with an
exponential stratum $H_r$.  If $\rho$ is an edge path of height $r$ and
$L_0>0$, then at least one of the following three possibilities occurs:
\begin{itemize}
\item $\rho f^M$ contains an $r$-legal segment of $r$-length greater
than $L_0$.
\item $\rho f^M$ contains fewer $r$-illegal turns than $\rho$.
\item $\rho f^M$ is a concatenation of indivisible Nielsen paths of
height $r$ and paths in $G_{r-1}$.
\end{itemize}
\end{lemnp}


\section{Improving nonexponential strata}\label{improvesec}

In \cite{tits1}, the authors improve the behavior of nonexponential
strata in a nonconstructive fashion.  We retrace some of their steps
here, replacing the nonconstructive parts by constructive arguments.

Let $H_r=\{E_r\}$ be a nonexponential stratum of a normalized train
track map $f\co G\rightarrow G$, and let $\rho$ be a path in $G_{r-1}$
originating at the terminal vertex of $E_r$.  We define a new map $f'\co
G'\rightarrow G'$ by removing $E_r$ and adding an edge $E_r'$ whose
initial vertex is the initial endpoint of $E_r$ and whose terminal
vertex is the terminal vertex of $\rho$.  We obtain $u_r'$ by tightening
$\bar{\rho}u_r(\rho f)$, so that $E_r'f'=E_r'u_r'$.  There is an obvious
homotopy equivalence $g\co G\rightarrow G'$ that sends $E_r$ to
$E_r'\bar{\rho}$.  With this marking, $f'$ induces the same outer
automorphism as $f$.  We say the $E_r'$ is obtained from $E_r$ by {\em
sliding along $\rho$}.

Let $\tilde{f}\co \tilde{G}\rightarrow \tilde{G}$ be a lift of $f$ that
fixes the initial endpoint of a lift $\tilde{E_r}$ of $E_r$.
Then $\tilde{f}$ leaves invariant a copy $H$ of the universal cover of
the connected component of $G_{r-1}$ that contains $u_r$.  Let
$h=\tilde{f}|_H$, and let $v_0$ be the terminal endpoint of
$\tilde{E}_r$.  Note that $v_0\in H$, and that $[v_0, v_0h]$ projects to
$u_r$.

\begin{lem}\label{constlem}
There exists a slide of $E_r$ to $E_r'$ with $E_r'f'=E_r'$ if and only
if $h$ fixes a point in $H$.
\end{lem}

\begin{proof}
If $h$ fixes $v\in H$, then sliding $E_r$ along $[v_0, v]$ yields a
fixed edge $E_r'$.  Conversely, if there exists a path $\rho$ such that
sliding $E_r$ along $\rho$ yields a fixed edge, then the terminal
endpoint of the lift of $\rho$ is fixed by $h$.
\end{proof}

In \secref{fixedsec}, we present an algorithm for detecting fixed points
of $h$.

\begin{lem}\label{nofixlem}
Assume that $h$ has no fixed points.  Let $U_{k}=[v_0, v_0h^k]$ and
$V_k=U_{k}\wedge U_{k+1}$ for $k\geq 0$.  Then $V_k$ is a proper
prefix of $V_{k+1}$.
\end{lem}

\begin{proof}
This follows from the discussion of preferred edges in the proof of
\cite[Proposition~5.4.3]{tits1}.
\end{proof}

As an immediate consequence of \lemref{nofixlem}, we obtain the
following lemma.
\begin{lemnp}\label{periodicpt}
If $h$ has a periodic point, then $h$ has a fixed point.
\end{lemnp}

The following proposition is the main result of this section; it
replaces a nonconstructive argument in \cite{tits1}.

\begin{prop}\label{makeeff}
Assume that $h$ has no fixed points.  We can compute a vertex in $v\in
H$ and an exponent $m\geq 1$ such that sliding $E_r$ along $[v_0, v]$
yields $E_r'(f^m)'=E_r'\cdot u_r'$ and $u_r'$ is a closed path starting
and ending at a fixed vertex.
\end{prop}

\begin{proof}
Let $v_k$ equal the terminal vertex of the path $V_k$
(\lemref{nofixlem}),\footnote{This agrees with our original definition
of $v_0$.} and let $w_k=[v_k, v_{k+1}]$.  The path $w_{k+m}$ is a
subpath of $w_kh^m$ for all $k, m\geq 0$.

The idea of the proof is to compute $w_0, w_1, w_2,\ldots, w_k$ until we
identify a suitable vertex $v$ in $w_k$.  Since $w_{k+1}$ is a subpath
of $w_kh$, we have $height(w_{k+1})\leq height(w_k)$, so that the height
of the paths $w_k$ has to stabilize eventually.  The following procedure
assumes that the height remains constant;  should the height drop while
the procedure is in progress, we simply start over.

Assume the height stabilizes at $r$.  This means that $H_r$ cannot be a
zero stratum.  Now, if $H_r$ is nonexponential, we have $|w_{k+1}|_r\leq
|w_k|_r$.  We keep iterating until we find $w_k$ such that
$|w_k|_r=|w_{k+1}|_r\geq 1$.  Let $v$ be the initial endpoint of an
occurrence of $E_r$ in $w_k$.  Then $v$ has the desired properties (and
we do not need to replace $f$ by a higher power in this case).

Now, assume that $H_r$ is exponential.  If we encounter a path $w_k$
that contains an $r$-legal subpath of $r$-length at least $2(C_f+1)$,
then $w_{k+1}$ contains a vertex $v$ that projects to a fixed vertex of
$f$ and whose $r$-distance from the closest $r$-illegal turn is at least 
$C_f$.  Now \lemref{bcc} yields that $v$ has the desired properties.

Assume that the length of $r$-legal subpaths remains bounded below
$2(C_f+1)$.  The number of illegal turns cannot go up and must stabilize
eventually, so that eventually we will end up in the third case of
\lemref{tricho} and see a composition of Nielsen paths of height $r$ and
paths in $G_{r-1}$.  We can detect this case in a brute-force fashion,
by checking all subpaths of $w_k$ in order to see whether they are
Nielsen.

Let $v$ be the initial point of one of the Nielsen paths.  Then $v$ is
periodic of period $m$, so that sliding $E_r$ along $[v_0, v]$ yields
the desired improvement of $f^m$.
\end{proof}

\begin{defn}\label{effdef}
Let $f\co G\rightarrow G$ be a normalized relative train track map with
a nonexponential stratum $H_r=\{E_r\}$.  We say that $H_r$ is {\em
efficient} if
\begin{enumerate}
\item $E_rf$ splits as $E_r\cdot u_r$ and $u_r$ is a closed path in
$G_{r-1}$,
\item if $u_r$ is a periodic Nielsen path, then its period is one (in
this case, we say that $E_r$ is {\em linear}), and
\item if $u_r$ is nontrivial, then there exists no slide of $E_r$ to
$E_r'$ such that $E_r'f'=E_r'$.
\end{enumerate}
We say that a relative train track map is {\em efficient} if it is
normalized, all its nonexponential strata are efficient, and the
nonexponential strata are sorted in such a way that if $u_r$ and $u_s$
are of the same height but $u_r$ is Nielsen and $u_s$ is not, then
$s>r$.
\end{defn}

\begin{lem}\label{linlem}
There exists a slide of $E_r$ to $E_r'$ with $E_r'f'=E_r'u_r'$ and
$u_r'$ a periodic Nielsen path if and only if $h$ commutes with a
nontrivial deck transformation.
\end{lem}

\begin{proof}
This lemma follows from \cite[Proposition~5.4.3]{tits1}.
\end{proof}

\begin{rem}
\lemref{linlem} implies that if $H_r$ is efficient and $u_r$ is
nontrivial and non-Nielsen, then there exists no slide that takes $u_r$
to a periodic Nielsen path.
\end{rem}

An infinite ray $\rho$ starting at a fixed vertex $v_0$ is a {\em fixed
ray} if $\rho f=\rho$.  It is {\em attracting} if there exists some $N$
such that if $\eta$ is a ray starting at $v_0$ and $|\rho\wedge\eta|>N$,
then $\eta f^n$ converges to $\rho$, i.e., $|\rho\wedge \eta f^n|$ goes
to infinity.  A {\em repelling} fixed ray is an attracting fixed ray for
a homotopy inverse of $f$.  See \cite{bdrydyn} for a detailed discussion
attracting and repelling fixed points for free group automorphisms.

\begin{lem}\label{attrlem}
Let $f\co G\rightarrow G$ be an efficient relative train
track map with a nonexponential stratum $H_r=\{E_r\}$ that is neither
linear nor constant.  Let
\[
R_r=E_ru_r(u_rf)(u_rf^2)\ldots
\]
Then $R_r$ is the unique attracting fixed ray of the form $E_r\gamma$,
for $\gamma\subset G_r$, and there are no Nielsen paths of the form
$E_r\gamma$.  In particular, we have $\lim_{k\rightarrow\infty}
\rho f^k=R_r$ for all basic paths $\rho$ of height $r$.
\end{lem}

\begin{proof}
This lemma follows from the proof of \cite[Lemma~5.5.1]{tits1}.  The
assumptions of \cite{tits1} are stronger that our assumptions, but a
close inspection of the proof shows that only our assumptions are needed
for the results that we use here.
\end{proof}

If $\rho$ is a path starting and ending at fixed points, then we can
find at most one path $\rho'$ with the same endpoints such that
$\rho'f=\rho$.  In this case, we write $\rho'=\rho f^{-1}$.  We define
$\rho f^{-k}$ in the obvious fashion.  If $\rho$ is closed, then $\rho
f^{-k}$ exists for all $k$.

\begin{lem}\label{replem}
Let $f\co G\rightarrow G$ be an efficient relative train
track map with a nonexponential stratum $H_r=\{E_r\}$ that is neither
linear nor constant.  Let
\[
S_r=E_r(\bar{u}_r f^{-1})(\bar{u}_r f^{-2})\ldots
\]
Then $S_r$ is the unique repelling fixed ray of the form $E_r\gamma$,
for $\gamma\subset G_r$.  In particular, we have
$\lim_{k\rightarrow\infty} \rho f^{-k}=S_r$ for all basic paths $\rho$ of
height $r$.
\end{lem}

\begin{proof}
\lemref{attrlem} implies that $h$ only has one repelling fixed ray.
Since $S_r$ is clearly fixed, it is the unique repelling fixed ray.
\end{proof}


\section{Detecting orbits of paths}\label{pathsec}

If $H_r$ is an exponential stratum and $\rho$ is a path of height $r$,
we let $\iota_r(\rho)$ equal the number of $r$-illegal turns in $\rho$.

\begin{lem}\label{detectnielsen}
Let $f\co G\rightarrow G$ be an efficient relative train track map.  If
$\rho$ is a circuit or edge path in $G$, then we can determine
algorithmically whether $\rho$ is a periodic Nielsen path; if $\rho$ is
Nielsen, then we can compute its period as well.
\end{lem}

\begin{proof}
Assume inductively that we can detect periodic Nielsen paths and
circuits in $G_{r-1}$.  We want to show that if $\rho$ is of height $r$,
then we can determine whether $\rho$ is Nielsen.

We first assume that $H_r=\{E_r\}$ is nonexponential.  Then $\rho$
splits as a concatenation of basic paths of height $r$ and paths in
$G_{r-1}$ (\lemref{pgsplit}), and it is Nielsen if and only if each of
these constituent paths is Nielsen.  Hence, we may assume that $\rho$ is
a basic path of height $r$, i.e., $\rho=E_r\gamma$ or
$\rho=E_r\gamma\bar{E}_r$ for some $\gamma\in G_{r-1}$.  If $E_rf=E_r$,
then $\rho$ is Nielsen if and only if $\gamma$ is Nielsen so that we are
done by induction.  If $E_r$ is neither constant nor linear, then
\lemref{attrlem} yields that $\rho$ cannot be Nielsen.

This leaves the case that $E_r$ is linear.  If $\rho=E_r\gamma$, then it
cannot be Nielsen (if $E_r\gamma$ were Nielsen, then \lemref{periodicpt}
would imply that we can slide $E_r$ to a constant edge, in violation of
efficiency of $f$).  Clearly, a path of the form $E_r\gamma\bar{E}_r$
can only be Nielsen if $\gamma$ is a (possibly negative) power of $u_r$,
which completes the proof for nonexponential $H_r$.

Now, assume that $H_r$ is exponential.  If an endpoint of $\rho$ is not
fixed, then $\rho$ cannot be Nielsen.  If both endpoints of $\rho$ are
fixed, we compute $\rho, \rho f, \rho f^2, \ldots$ until one of the
following three cases occurs:
\begin{itemize}
\item We encounter some image $\rho f^k$ that contains an $r$-legal path
whose length exceeds $2C_f$.  Then \lemref{critlen} implies that $\rho$
is not Nielsen.
\item We encounter some image $\rho f^k$ that contains fewer $r$-illegal
turns than $\rho$.  Since $f$ does not increase the number of
$r$-illegal turns, $\rho$ is not Nielsen.
\item We can express $\rho$ as
$\rho=\alpha_1\beta_1\alpha_2\beta_2\cdots\alpha_m\beta_m$, where the
$\alpha_i$ are Nielsen paths of height $r$, and the $\beta_i$ are
subpaths in $G_{r-1}$, such that we encounter some $\rho
f^k=\alpha_1(\beta_1 f^k)\cdots  \alpha_m(\beta_m f^k)$.  In this case,
$\rho$ is Nielsen if and only if the $\beta_i$ are Nielsen.
\end{itemize}

One of these three cases must occur eventually, and we can detect the
third case in a brute-force way by checking all possible decompositions
of $\rho$.

Finally, if $H_r$ is a zero stratum, then $\rho$ cannot possibly be
Nielsen, so that the proof is complete.
\end{proof}

If $u$ is a closed path and $\rho$ is an arbitrary edge path, we let
$p_u(\rho)$ equal the largest exponent $m$ so that $u^m$ is a prefix of
$\rho$.

\begin{lem}\label{nielsengrowth}
Let $f\co G\rightarrow G$ be a relative train track map with an
exponential stratum $H_r$ and a closed Nielsen path $u$ of height $r$.
If $\rho$ is an edge path of height $r$ and $k\geq 0$ an exponent such
that $p_u(\rho)=m$ and $p_u(\rho f^k)=l$, then $\iota_r(\rho)\geq
(2m-l-1)\iota_r(u)$.
\end{lem}

\begin{proof}
We express $\rho$ as $\rho=u^m\gamma$.  Since we have $p_u(\rho f^k)=l$,
we conclude that $p_{\bar u}(\gamma f^k)\geq m-l-1$, so that
$\iota_r(\gamma f^k)\geq (m-l-1)\iota_r(u)$.  Since $f$ does not
introduce new illegal turns, we have $\iota_r(\gamma)\geq
(m-l-1)\iota_r(u)$, so that $\iota_r(\rho)\geq (2m-l-1)\iota_r(u)$.
\end{proof}

\begin{lem}\label{detectgrowth}
Let $f\co G\rightarrow G$ be an efficient train track map and let
$\rho$ be a non-Nielsen path whose endpoints are fixed.  Then for any
$L>0$, we can compute an exponent $k_0>0$ such that $|\rho f^k|>L$
and $|\rho f^{-k}|>L$ (if $\rho f^{-k}$ exists) for all $k\geq k_0$.
\end{lem}

\begin{proof}
We assume inductively that the lemma holds for the restriction of $f$ to
$G_{r-1}$.  We first assume that $H_r=\{E_r\}$ is a nonexponential
stratum.  Then $\rho$ splits as a concatenation of basic paths of height
$r$ and paths in $G_{r-1}$, so that we may assume that $\rho$ is a
non-Nielsen basic path of height $r$, i.e., $\rho=E_r\gamma\bar{E}_r$ or
$\rho=E_r\gamma$.

Assume that $E_r$ is neither constant nor linear.  Then we can find a
prefix $R$ of $R_r$ as well as a prefix $S$ of $S_r$ (see
\lemref{attrlem} and \lemref{replem}) of length greater than $L$ for
which $|Rf|>|R|+C_f$ and $|Sf^{-1}|>|S|+C_f$.  Now \lemref{attrlem},
\lemref{replem}, and \lemref{bcc} imply that we can find some exponent
$k_0$ such that $R$ is a prefix of $\rho f^k$ and $S$ is a prefix of
$\rho f^{-k}$ for all $k\geq k_0$.  We conclude that $|\rho f^{\pm
k}|>L$ for all $k\geq k_0$.

If $E_r$ is constant, then the inductive hypothesis applied to $\gamma$
completes the proof.  This leaves the case that $E_r$ is linear.  Let
$s$ be the height of $\gamma$.  If $s$ is smaller than the height of
$u_r$, we conclude that no copy of $u_r$ will cancel completely in $\rho
f^k$ for any $k>0$, so that we have $|\rho f^{\pm k}|>L$ for all $k>L$.

If $s$ equals the height of $u_r$ and $H_s$ is nonexponential, then no
more than $|\gamma|$ copies of $u_r$ cancel in $|\rho f^k|$, so that we
have $|\rho f^{\pm k}|>L$ for all $k>L+|\gamma|$.  If $H_s$ is
exponential, then for all $k\geq 0$, the number of copies of $u_r$ that
cancel in $\rho f^k$ is bounded by $\iota_s(\gamma)$, so that $|\rho
f^k|>L$ if $k>L+\iota_s(\gamma)$.

We still need to study the length of $\rho f^{-k}$ for $k\geq 0$.  Let
$m=p_{u_r}(\gamma f^{-k})$ and $l=p_{u_r}(\gamma)$.  Then
\lemref{nielsengrowth} implies that $\iota_r(\gamma f^{-k})\geq
(2m-l-1)\iota_r(u_r)$.  This implies that $\iota_r(\rho f^{-k})\geq
k\iota_r(\bar{u}_r)+(2m-l-1)\iota_r(u_r)-2m\iota_r(u_r)=(k-l-1)\iota_r(u_r)$,
so that $|\rho f^{-k}|>L$ if $k>L+l+1$.

If $s$ exceeds the height of $u_r$, then, by definition of efficiency,
$H_s$ is also linear, and $\rho$ splits into subpaths of the form
$E_r\eta$, $E_s\eta$,  and $E_r\eta\bar{E}_s$, where $\eta\subset
G_{s-1}$.  The first two cases are done by induction on $s$, so that we
only need to consider the case $E_r\eta\bar{E}_s$.  This case is
essentially the same as the previous one (we need to apply
\lemref{nielsengrowth} to both $\eta$ and $\bar{\eta}$), except we need
to consider the possibility that there is a closed Nielsen path $\tau$
such that $u_r=\tau^a$, $u_s=\tau^b$, and $\eta=\tau^c$.  In this case,
we have $a\neq b$ (or else $E_r\bar{E}_s$ would be Nielsen, in violation
of efficiency), so that $|(E_r\eta\bar{E}_s)f^k|\geq k-c$, so that
$|(E_r\eta\bar{E}_s)f^k|>L$ if $k>L+c$.

Finally, assume that $H_r$ is exponential.  In this case, we compute
$\rho, \rho f, \ldots$ until we either find some $k_0$ such that $\rho
f^{k_0}$ has an $r$-legal subpath of $r$-length greater than $L+2C_f$
(in which case \lemref{critlen} yields that $|\rho f^k|>L$ for all
$k\geq k_0$), or, by \lemref{tricho}, we encounter some $k$ such that
$\rho f^k$ is a composition of indivisible Nielsen paths of height $r$
and paths in $G_{r-1}$.  Since $\rho$ is non-Nielsen, one of the
subpaths in $G_{r-1}$ must be non-Nielsen, so that we are done by
induction.

In order to understand lengths under backward iteration, we need to
consider two cases:  If $\rho$ is not a composition of indivisible
Nielsen paths of height $r$ and paths in $G_{r-1}$, then
\lemref{tricho} implies that the number of $r$-illegal turns has to go
up under backward iteration.  In this case, we simply compute $\rho
f^{-1}$, $\rho f^{-2}, \ldots$ until we find some $k_0$ for which $\rho
f^{-k_0}$ contains $L$ $r$-illegal turns, and we conclude that $|\rho
f^{-k}|>L$ for all $k\geq k_0$.

If $\rho$ is a concatenation of indivisible Nielsen paths of height $r$
and paths in $G_{r-1}$, then one of the subpaths $\gamma$ in $G_{r-1}$
is not Nielsen, so that the inductive hypothesis applies to $\gamma$.
\lemref{fppreim} guarantees that $\gamma f^{-k}$ exists for all $k\geq
0$, so that we are done.
\end{proof}

\begin{prop}\label{detectpathorbits}
Let $f\co G\rightarrow G$ be an efficient train track map, and let
$\rho_1$ and $\rho_2$ be paths whose endpoints are fixed.  Then we can
determine algorithmically whether $\rho_2$ is the image of $\rho_1$
under some power of $f^k$, and we can compute the exponent $k$ if it
exists.
\end{prop}

\begin{proof}
Using \lemref{detectnielsen}, we determine whether $\rho_1$ is a
periodic Nielsen path.  If it is, we simply enumerate all distinct images
of $\rho_1$ and check whether $\rho_2$ is among them.  If $\rho_1$ is
not Nielsen, we apply \lemref{detectgrowth} with $L=|\rho_2|$ to obtain
an exponent $k_0$.  Now we compute $\rho, \rho f, \ldots, \rho_1f^{k_0}$
and check whether $\rho_2$ is contained in this list.

If $\rho_2$ is contained in this list, we obtain a positive answer as
well as the desired exponent $k$.  If not, we switch $\rho_1$ and
$\rho_2$ and repeat the argument.
\end{proof}

\begin{thm}\label{computenielsen}
Let $f\co G\rightarrow G$ be an efficient train track map with an
exponential stratum $H_r$.  Then we can compute all indivisible periodic
Nielsen paths of height $r$ as well as their periods.
\end{thm}

\begin{proof}
Let $\alpha$ be an indivisible Nielsen path of height $r$.  Then
$\alpha$ contains exactly one $r$-illegal turn, and the $r$-length of
its two $r$-legal subpaths is bounded by $C_f$ (\lemref{bcc}).
Moreover, the first and last (possibly partial) edges of $\alpha$ are
contained in $H_r$.

For an edge $E$ in $H_r$, let $P_E$ be the set of maximal subpaths in
$G_{r-1}$ of $Ef$, and let $P=\cup_{E\in H_r} P_E$.  If $\beta$ is a
maximal subpath in $G_{r-1}$ of $\alpha$, then there exists some
$\gamma\in P$ and $k\geq 0$ such that $\beta=\gamma f^k$.

Let $\gamma$ be a path in $P$.  If $\gamma$ is Nielsen, we let
$L_\gamma=\max_{k} \{|\gamma f^k|\}$.  If $\gamma$ is not Nielsen,
\lemref{detectgrowth} with $L=C_f$ yields an exponent $k_0$ such that
$|\rho f^k|>L$ for all $k\geq k_0$.  We let $L_\gamma=\max_{0\leq k<k_0}
\{|\rho f^k|\}$.

Let $M=\max_{\gamma\in P} \{L_\gamma\}$ and observe that $\alpha$ has no
subpaths in $G_{r-1}$ whose length exceeds $M$.  Let $Q$ be the set of
all edge paths $\rho$ such that $\rho$ contains exactly one $r$-illegal
turn, the $r$-length of $r$-legal subpaths is bounded by $C_f$, the
length of subpaths in $G_{r-1}$ is bounded by $M$, and the first and
last edges are contained in $H_r$.  Clearly, if $\alpha$ is an
indivisible Nielsen path of height $r$, then $\alpha$ is a subpath of
some $\rho\in Q$.

We define a map $g\co Q\rightarrow G\cup \{\ast\}$\footnote{$\ast$ is
merely some termination symbol.} by letting $\rho g$ equal the unique
maximal subpath of $\rho f$ contained in $Q$ if $\rho f$ contains an
$r$-illegal turn, and we let $\rho g=\ast$ if $\rho f$ contains no
$r$-illegal turn.

For each $\rho\in G$, we compute $\rho, \rho g, \rho g^2, \ldots$ until
we either encounter $\ast$ (in which case $\rho$ has no Nielsen subpath)
or we find that $\rho g^k=\rho g^m$ for some $0\leq k<m$.  Then $\rho
g^k$ contains an indivisible Nielsen subpath $\alpha$, and we can easily
compute the endpoints of $\alpha$.  Moreover, if $k$ and $m$ are as
small as possible, then $m-k$ is the period of $\alpha$.  Since all
indivisible Nielsen paths of height $r$ show up in this fashion, the
proof is complete.
\end{proof}

\begin{cor}
Given an efficient relative train track map $f\co G\rightarrow G$, we
can compute an exponent $k\geq 1$ such that all periodic Nielsen paths of
$f^k$ have period one.
\end{cor}


\section{Detecting fixed points}\label{fixedsec}

Let $f\co G\rightarrow G$ be a normalized relative train track map with
a nonexponential stratum $H_r=\{E_r\}$.  Assume that the restriction of
$f$ to $G_{r-1}$ is efficient.  The purpose of this section is to
present an algorithm for determining whether $E_r$ has a slide to a
constant edge (\propref{findslide}).  This is the last missing piece in
our computation of efficient maps (\thmref{trainsummary}).

We have $E_rf=E_ru_r$, and we want to express $u_r$ as the path obtained
by tightening $\bar{\rho}(\rho f)$ for some path $\rho$ in $G_{r-1}$, if
possible.  To this end, choose a fixed vertex $\bar v_0\in G_{r-1}$.
The main idea is to perform a breadth-first search of edge paths $\rho$
originating at $\bar v_0$, keeping track of the paths obtained by
tightening $\bar{\rho}(\rho f)$ until we either encounter $u_r$ or we
determine that further searching will not yield $u_r$.  If we encounter
$u_r$ along the way, then sliding $E_r$ along $\bar{\rho}$ will turn it
into a constant edge.

It will be convenient to work in the universal cover $H$ of $G_{r-1}$,
constructing partial lifts $h$ of $f$ as we go along
(\consref{covercons}), beginning with $T_0=U_0=\{v_0\}$.  For a vertex
$v$ in $H$, we define $\rho_v$ to be the path $[v_0, v]$ and $w_v$ to be
the projection of $[v, vh]$.  Note that $w_v$ is the projection of the
path obtained by tightening $\bar{\rho}_v(\rho_v h)$.

\begin{figure}
\renewcommand{\epsfsize}[2]{0.3\textwidth}
\begin{center}
\psfrag{v0}{$v_0$}
\psfrag{v}{$v$}
\psfrag{vh}{$vh$}
\psfrag{Mv}{$M_v$}
\psfrag{Nv}{$N_v$}
\epsfbox{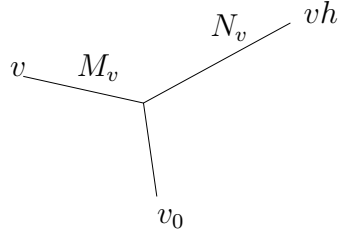}
\end{center}
\caption{Looking for fixed points}\label{searchpic}
\end{figure}

We let
$M_v=|v_0, v|-|[v_0, v]\wedge [v_0, vh]|$ and
$N_v=|v_0, vh|-|[v_0, v]\wedge [v_0, vh]|$ (\figref{searchpic}).
Note that $|w_v|=M_v+N_v$.

The following is a partial list of conditions under which we need not
extend our search beyond a vertex $v$:
\begin{itemize}
\item The path $w_v$ was encountered before in our search.  In this
case, searching beyond $v$ will not yield any new results.
\item If $|w_v|>|u_r|+C_f$, $M_{v'}>0$ and $N_{v'}>C_f$ for some vertex
$v'\in [v_0, v]$, then \lemref{bcc} implies that $|w_{v'}|>|u_r|$ for
all vertices $v'$ beyond $v$, so that we will not encounter $u_r$ if we
search beyond $v$.
\end{itemize}

Assume that there exists an infinite sequence $v_0, v_1, v_2, \ldots$
such that $v_k\neq v_{k+2}$, $|v_k, v_{k+1}|=1$ for all $k$, and none of
the two cases above occurs.  Then $|w_{v_k}|$ goes to infinity (or else
there would be some repetition along the way), and we have $M_{v_k}=0$
or $N_{v_k}\leq C_f$ for all $k$.  In fact, we have $M_{v_k}=0$ for all
$k$ or $N_{v_k}\leq C_f$ for all $k$ (otherwise we would encounter a
fixed interior vertex, i.e., a vertex $v_k\neq v_0$ for which $w_{v_k}$
is trivial, so that we would have reached our first termination
criterion because $w_{v_0}$ is trivial).  In the first case, the $v_k$
define an attracting fixed ray of $h$.  In the second case, they define
a repelling fixed ray of $h$.

\subsection{Attracting fixed rays}

If $v_0, v_1, v_2,\ldots$ is an attracting fixed ray with no interior
fixed vertices, then this sequence is determined by $v_0$ and $v_1$ alone
because the first edge of $[v_k, v_kh]$ is the same as the edge $[v_k,
v_{k+1}]$; otherwise we would encounter a trivial $w_k$ along the way.
For the same reason, the edge $[v_0, v_1]$ cannot project to a constant
edge.  In other words, we need to consider at most one attracting fixed
ray for each nonconstant edge originating at $v_0$, and we can easily
compute arbitrarily long prefixes of each ray.

In order to determine when to stop following an attracting ray, we will
identify some $k_0$ such that $|v_k, v_kh|>|u_r|+C_f$ for all $k\geq
k_0$.  This implies that $|w_{v_k}|>|u_r|+C_f$ for all $k\geq k_0$.
Moreover, if $v$ is a vertex such that $v_{k_0}\in[v_0, v]$, then
\lemref{bcc} implies that $|w_v|>|u_r|$, so that we can terminate our
search at $v_{k_0}$.

First, assume that the edge bounded by $v_0$ and $v_1$ is contained in
an exponential stratum $H_s$.  Then $[v_0, v_k]$ projects to an
$r$-legal path for all $k$, and we have $|v_k, v_kh|_r\geq |v_0, v_k|_r$
because $f$ is normalized.  Hence, we only need to compute $v_0, \ldots,
v_k$ until the $r$-length of $[v_0, v_k]$ exceeds $|u_r|+C_f$.

Now, assume that $[v_0, v_1]$ projects to a nonexponential edge $E_s$.
Since $v_0, v_1, \ldots$ is a fixed ray, $[v_0, v_1]$ cannot project to
$\bar{E}_s$, and so $\lim_{k\rightarrow\infty} [v_0, v_k]$ equals $R_s$.
If $E_s$ is linear, then we reach our first termination criterion after
at most $|u_s|$ steps, so that we may assume that $E_s$ is neither
constant nor linear.

\begin{lem}\label{interpolate}
Let $L>0$ and assume that $v$ is a vertex in $H$ such that $|v, vh|\geq
L$, $|vh, vh^2|\geq L$, and $vh\in [v, vh^2]$.  Then, for all $x\in [v,
vh]$, we have
\[
|x, xh|\geq \frac{2L}{K_f+1}-D_f.
\]
\end{lem}

\begin{proof}
Let $t=|x, v|$.  Then \ineqref{qi} implies that $|xh, vh|\geq
\frac{t}{K_f}-D_f$ and $|xh, vh^2|\leq K_f(L-t)+D_f$.  We conclude that
$|x, xh|\geq L-t+\max\{\frac{t}{K_f}-D_f, L-K_f(L-t)-D_f\}$.  The
minimum of the right-hand side of this inequality is attained for
$t=\frac{LK_f}{K_f+1}$, and substituting this value yields a lower bound
of$\frac{2L}{K_f+1}-D_f$.
\end{proof}

We choose $L$ such that $\frac{2L}{K_f+1}-D_f>|u_r|+C_f$.  Now
\lemref{detectgrowth} yields an exponent $k_0$ such that $|u_s f^k|>L$
for all $k\geq k_0$.  We only need to compute $v_0, \ldots, v_k$ until
$[v_0, v_k]$ projects to $E_su_s\cdots (u_s f^{k_0})$, and
\lemref{interpolate} guarantees that $|w_v|>|u_r|+C_f$ for all $v$
beyond $v_k$.  This completes our algorithm in the case of attracting
fixed rays.

\subsection{Repelling fixed rays}

In the attracting case, we construct fixed rays edge by edge, and an
attracting fixed ray that contains no interior fixed points is
determined by its first edge.  In the repelling case, the situation is
more complicated, but the following lemma still give us a way of
computing successive edges in potential fixed rays given a sufficiently
long prefix.

\begin{lem}\label{computerepel}
Let $v_0, v_1,\ldots, v_k$ be a sequence such that $N_{v_j}\leq C_f$ for
all $0\leq j\leq k$ and $M_{v_k}>C_f$.   Then at most one vertex $v$
adjacent to $v_k$, other than $v_{k-1}$, can be contained in a repelling
ray originating at $v_0$, and we can find $v$ algorithmically or
determine that there is no such $v$.  Moreover, if $v'$ is a vertex
satisfying $v_k\in [v_0, v']$ and $v\not\in [v_0, v']$, then $M_{v'}\geq
M_{v_k}+|v_k, v'|-C_f$.
\end{lem}

\begin{figure}
\renewcommand{\epsfsize}[2]{0.4\textwidth}
\begin{center}
\psfrag{v0}{$v_0$}
\psfrag{vk}{$v_k$}
\psfrag{vkh}{$v_kh$}
\psfrag{L}{$L$}
\psfrag{Cf}{$\leq C_f$}
\epsfbox{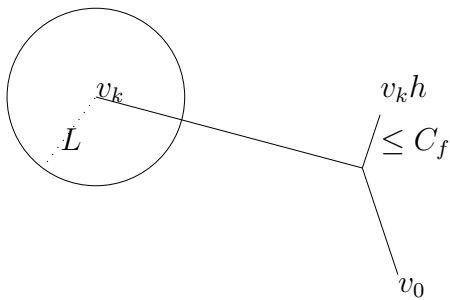}
\end{center}
\caption{Finding repelling fixed rays}\label{repelpic}
\end{figure}

\begin{proof}
Using \ineqref{qi}, we find some $L>0$ such that if $\rho$ is a path of
length at least $L$, then $|\rho f|\geq 2C_f+1$.  Now we enumerate all
vertices $p_1,\ldots, p_m$ such that $|v_k, p_i|=L$ and $v_k\in [v_0,
p_i]$ for all $i$ (\figref{repelpic}).  \lemref{bcc} yields that $|[v_k,
p_i]h \wedge [v_k, p_j]h|\leq C_f$ if $|[v_k, p_i]\wedge [v_k, p_j]|=0$.

If $p_i$ and $p_j$ are contained in fixed rays, then $N_{p_i}<C_f$ and
$N_{p_j}<C_f$.  This implies that $|[v_k, p_i]h\wedge [v_k, p_j]h|>C_f$,
so that $|[v_k, p_i]\wedge [v_k, p_j]|>0$.  Hence, if there exists some
$p_i$ such that $|[v_k, p_i]h\wedge [v_k, v_kh]|>C_f$, then the second
vertex $v$ in $[v_k, p_i]$ is uniquely determined by this property.

The last claim is an immediate consequence of \lemref{bcc}.
\end{proof}

Another complication in the repelling case is that the height may go up
as we apply \lemref{computerepel} to compute subsequent vertices.  The
following lemmas provide a means of handling this possibility.

\begin{lem}\label{lowprefix}
Assume that $H_s$ is an exponential stratum and let
$C=S_f(1+\#\ee{G})$.  If $\eta$ is a repelling fixed ray of height
$s$ with a maximal prefix $\alpha$ in $G_{s-1}$, then $|\alpha f|+C\geq
|\alpha|.$
\end{lem}

\begin{proof}
If the initial vertex $v_0$ is contained in a contractible component of
$G_{s-1}$, then the claim is trivial, so that we may assume that $v_0$
is contained in a noncontractible component of $G_{s-1}$.  By
\lemref{fixedvertices}, the terminal endpoint of $\alpha$ is fixed.

Choose $\beta$ so that $\eta=\alpha\beta$.  By definition, the first
edge in $\beta$ is contained in $H_s$.  Let $\gamma$ be the maximal
subpath in $G_{s-1}$ of $\beta f$.  It suffices to show that
$|\gamma|\leq C$.

If $\gamma$ is a subpath of $Ef$ for some edge $E\subset H_s$, then
$|\gamma|\leq S_f$.  If $\gamma$ is the image of some subpath
$\gamma'\subset G_{s-1}$ of $\beta$, then \lemref{fixedvertices} implies
that $\gamma'$ is contained in a contractible component of
$G_{s-1}$,\footnote{ Otherwise $\beta$ would have an initial subpath
$\eta$ of height $s$, starting and ending at fixed vertices, so that
$\eta f$ is trivial.  This is impossible because $f$ is a homotopy
equivalence.} so that $|\gamma'|\leq \#\ee{G}$.  This implies that
$|\gamma|\leq S_f\#\ee{G}$.
\end{proof}

\begin{lem}\label{illegalgrowth}
If $H_s$ is an exponential stratum and the sequence $v_0, v_1,\ldots$
defines a repelling fixed ray $\eta$ of height $s$ without interior
fixed points, then $\iota_s(w_{v_k})$ is an unbounded nondecreasing
function of $k$.
\end{lem}

\begin{proof}
Since $\eta$ has no interior fixed points, it cannot be a concatenation
of Nielsen paths of height $r$ and subpaths in $G_{r-1}$.  This implies
that $\eta$ contains infinitely many $r$-illegal turns.  Now
\lemref{critlen} implies that the distance between two $r$-illegal turns
is bounded by some constant $L$.  Since $\eta$ is repelling, $|w_{v_k}|$
is unbounded, which proves the claim.
\end{proof}

\begin{lem}\label{nonexponentialray}
Assume that $H_s$ is a nonexponential stratum and that $\eta$ is a
repelling fixed ray of height $s$ with no fixed interior vertices.  Then
$\eta=S_s$.
\end{lem}

\begin{proof}
This is an immediate consequence of \lemref{pgsplit} and
\lemref{replem}.
\end{proof}

\lemref{nonexponentialray} implies that if the height goes up as we
follow a potential repelling fixed ray, then the height must eventually
stabilize at an exponential stratum.

We now continue our breadth-first traversal of vertices in $H$.  If we
encounter a vertex $v_k$ such that $[v_0, v_k]$ satisfies the hypotheses
of \lemref{computerepel}, then we need to consider the possibility that
$[v_0, v_k]$ is a prefix of a repelling fixed ray.  In this case, we
use \lemref{computerepel} to compute subsequent vertices $v$.  (In this
process, $M_v$ may drop below $C_f$, so that \lemref{computerepel} no
longer applies; in this case, we simply continue our breadth-first
search.  This is not a problem, however, because it can only happen
finitely many times before we encounter our first termination
criterion.)

Let $s$ be the height of the potential repelling ray computed so far.
If $H_s$ is nonexponential, then our ray must converge to $S_s$.  Using
arguments similar to those in the attracting case, we follow $S_s$ until
we recognize a vertex $k_0$ such that for all vertices $v$ beyond
$v_{k_0}$, we have $M_v>\max\{C, |u_r|\}$ (where $C$ is the constant from
\lemref{lowprefix}).  $M_v>C$ guarantees that we are not following a
prefix of a ray of greater height, and $M_v>|u_r|$ implies that we will
not encounter $u_r$ as we follow the ray.

If $H_s$ is exponential, then we follow our ray until we encounter a
vertex $v$ for which $\iota_s(w_v)>\max\{C, |u_r|\}$.  Once again,
\lemref{lowprefix} guarantees that the height will not go up if we
continue following our ray, and we will not encounter $u_r$ if we
continue our search.  Hence, our algorithm terminates in all possible
cases.

\subsection{Picking up the pieces}

\begin{prop}\label{findslide}
If $H_r=\{E_r\}$, then we can determine algorithmically whether there
exists a path $\rho\subset G_{r-1}$ such that $u_r$ is obtained by
tightening $\bar{\rho}(\rho f)$, and we can compute $\rho$ if it exists.
\end{prop}

\begin{proof}
If $\rho$ exists, then its initial vertex is a fixed vertex in
$G_{r-1}$.  Repeating the procedure above for each fixed vertex in
$G_{r-1}$ yields the desired algorithm.
\end{proof}

\begin{thm}\label{trainsummary}
Given an outer automorphism ${\cal O}$ of $F_n$, we can compute a
efficient relative train track map $f\co G\rightarrow G$ as well as an
exponent $k\geq 1$ such that $f$ represents ${\cal O}^k$.
\end{thm}

\begin{proof}
We can compute an exponent $k\geq 1$ and a normalized relative train
track map $f\co G\rightarrow G$ representing ${\cal O}^k$.  Now we
assume inductively that the restriction of $f$ to $G_{r-1}$ is
efficient.  If $H_r$ is zero or exponential, then there is
nothing to do.  If $H_r=\{E_r\}$ is nonexponential, then we first use
\propref{findslide} to determine whether there exists a slide of $E_r$
to a constant edge.  If no such edge exists, we use \propref{makeeff} to
achieve efficiency of $H_r$.
\end{proof}


\section{Proof of the main result}\label{mainsec}

\begin{lem}\label{circuitgrowth}
Let $f\co G\rightarrow G$ be an efficient relative train track map.
There exists an algorithm that, given a circuit $\sigma$ in $G$ and a
constant $L>0$, determines whether $\sigma$ is Nielsen.  If $\sigma$ is
not Nielsen, then the algorithm finds an exponent $k_0$ such that
$|\sigma f^k|>L$ for all $k\geq k_0$.
\end{lem}

\begin{proof}
\lemref{detectnielsen} takes care of the detection of Nielsen circuits.
If $\sigma$ is not Nielsen, then we consider the height $r$ of $\sigma$.
If $H_r$ is nonexponential, then it splits as a concatenation of basic
paths of height $r$ (\lemref{pgsplit}), so that \lemref{detectgrowth}
completes the proof in this case.

If $H_r$ is exponential, then we compute $\sigma, \sigma f, \sigma
f^2,\ldots$ until we encounter an image $\sigma'=\sigma f^k$ for some
$k>0$ such that $\sigma'$ contains an $r$-legal path of length greater than
$2(C_f+1)$ or $\sigma'$ is a concatenation of Nielsen paths of height
$r$ and paths in $G_{r-1}$.

We can recognize both possibilities algorithmically.  In the first case,
$\sigma' f$ splits at a fixed vertex in a long $r$-legal subpath.  In
the second case, $\sigma'$ splits at the terminal endpoint of a subpath
in $G_{r-1}$.  In either case, \lemref{detectgrowth} completes the
proof.
\end{proof}

\begin{thm}\label{mainalg1}
Let $\phi$ be an automorphism of $F_n$.  The exists an algorithm
that, given two elements $u, v\in F_n$, determines whether there exists
some exponent $N$ such that $u\phi^N$ is conjugate to $v$.  If such an
$N$ exists, then the algorithm will compute $N$ as well.
\end{thm}

\begin{proof}
\thmref{trainsummary} yields an exponent $k$ and an efficient relative
train track map $f\co G\rightarrow G$ that represents the outer
automorphism defined by $\phi^k$.  We can find some constant $Q\geq 1$ such
that if $\sigma$ is a circuit in $G$ representing a conjugacy class
$\omega$ in $F_n$, then $\frac1Q|\omega|\leq |\sigma|\leq
Q|\omega|$.\footnote{The length of a conjugacy class $\omega$ is
defined to be the length of the shortest element in $\omega$.}

Represent the conjugacy class of $u$ by a circuit $\sigma$.  If $\sigma$
is a Nielsen circuit of period $p$, then we conclude that $u\phi^{kp}$
is conjugate to $u$.  Now we compute $u, u\phi,\ldots, u\phi^{kp-1}$ and
check whether any conjugate of $v$ is in this list.

If $\sigma$ is not Nielsen, we let $L=Q\cdot S_\phi^k\cdot |v|$, and we
find some exponent $k_0$ such that $|\sigma f^j|>L$ for all $j\geq k_0$.
We conclude that the length of the conjugacy class of $u\phi^j$ exceeds
$|v|$ for all $j\geq kk_0$.  Now we list $u, u\phi, u\phi^2,\ldots,
u\phi^{kk_0-1}$ and check whether any conjugate of $v$ is in this list.
If no conjugate is contained in this list, then we exchange $u$ and $v$
and repeat the argument.  This completes the proof.
\end{proof}

\begin{thm}\label{mainalg2}
Let $\phi$ be an automorphism of $F_n$.  The exists an algorithm that,
given two elements $u, v\in F_n$, determines whether there exists some
exponent $N$ such that $u\phi^N=v$.  If such an $N$ exists, then the
algorithm will compute $N$ as well.
\end{thm}

\begin{proof}
We use a trick from \cite{bfgafa2}. Let $F'=F_n\ast\langle
a\rangle$ and define $\psi\in Aut(F')$ by letting $x\psi=x\phi$ if $x\in
F_n$ and $a\psi=a$.  If $w\in F_n$, then $wa$ is cyclically reduced in
$F'$, so that $u\phi^N=v$ if and only if $(ua)\psi^N$ is conjugate to
$va$.  Now \thmref{mainalg1} completes the proof.
\end{proof}

\bibliographystyle{alpha}
\bibliography{pb}
\par

{\noindent \sc Department of Mathematics\\
160 Convent Ave., NAC 8/133\\
New York, NY 10031, USA\\}
{\noindent \it E-mail:} brinkman@sci.ccny.cuny.edu

\end{document}